\documentclass[12pt,letterpaper,final,twoside,leqno]{amsart}
\usepackage{eucal, mathrsfs}
\usepackage{amsmath,amssymb,amsthm,
amscd,amsopn}
\usepackage{hyperref}
\usepackage{times}
\usepackage{xspace}
\usepackage{epsfig,epic,eepic,latexsym,color}
\usepackage[all]{xy}
\usepackage{paralist}
\usepackage{stmaryrd}
\usepackage{enumitem}
\usepackage{color}
\usepackage{soul}

\catcode`~=11 \def\UrlSpecials{\do\~{\kern -.15em\lower .7ex\hbox{~}\kern .04em}} \catcode`~=13

\newcommand{\urlwofont}[1]{\urlstyle{same}\url{#1}}

\newcounter{are-there-sections}
\makeatletter

\renewcommand\subsection{
  \renewcommand{\sfdefault}{pag}
  \@startsection{subsection}%
  {2}{0pt}{-\baselineskip}{.2\baselineskip}{\raggedright
    \sffamily\itshape\small
  }}
\renewcommand\section{
  \renewcommand{\sfdefault}{phv}
  \@startsection{section} %
  {1}{0pt}{\baselineskip}{.2\baselineskip}{\centering
    \sffamily
    \scshape
}}
\newcounter{lastyear}
\addtocounter{lastyear}{-1}

\newcommand\noin{\noindent}

\newtheoremstyle{bozont}{3pt}{3pt}%
     {\itshape}
     {}
     {\bfseries}
     {.}
     {.5em}
     {\thmname{#1}\thmnumber{ #2}\thmnote{ \rm #3}}
\newtheoremstyle{bozont-sf}{3pt}{3pt}%
     {\itshape}
     {}
     {\sffamily}
     {.}
     {.5em}
     {\thmname{#1}\thmnumber{ #2}\thmnote{ \rm #3}}
\newtheoremstyle{bozont-sc}{3pt}{3pt}%
     {\itshape}
     {}
     {\scshape}
     {.}
     {.5em}
     {\thmname{#1}\thmnumber{ #2}\thmnote{ \rm #3}}
\newtheoremstyle{bozont-remark}{3pt}{3pt}%
     {}
     {}
     {\scshape}
     {}
     {.5em}
     {\thmname{#1}\thmnumber{ #2.}\thmnote{\ \  ( #3)}}
\newtheoremstyle{bozont-def}{3pt}{3pt}%
     {}
     {}
     {\bfseries}
     {.}
     {.5em}
     {\thmname{#1}\thmnumber{ #2}\thmnote{ \rm #3}}
\newtheoremstyle{bozont-reverse}{3pt}{3pt}%
     {\itshape}
     {}
     {\bfseries}
     {.}
     {.5em}
     {\thmnumber{#2.}\thmname{ #1}\thmnote{ \rm #3}}
\newtheoremstyle{bozont-reverse-sc}{3pt}{3pt}%
     {\itshape}
     {}
     {\scshape}
     {.}
     {.5em}
     {\thmnumber{#2.}\thmname{ #1}\thmnote{ \rm #3}}
\newtheoremstyle{bozont-reverse-sf}{3pt}{3pt}%
     {\itshape}
     {}
     {\sffamily}
     {.}
     {.5em}
     {\thmnumber{#2.}\thmname{ #1}\thmnote{ \rm #3}}
\newtheoremstyle{bozont-remark-reverse}{3pt}{3pt}%
     {}
     {}
     {\sc}
     {.}
     {.5em}
     {\thmnumber{#2.}\thmname{ #1}\thmnote{ \rm #3}}
\newtheoremstyle{bozont-def-reverse}{3pt}{3pt}%
     {}
     {}
     {\bfseries}
     {.}
     {.5em}
     {\thmnumber{#2.}\thmname{ #1}\thmnote{ \rm #3}}
\newtheoremstyle{bozont-def-newnum-reverse}{3pt}{3pt}%
     {}
     {}
     {\bfseries}
     {}
     {.5em}
     {\thmnumber{#2.}\thmname{ #1}\thmnote{ \rm #3}}
\theoremstyle{bozont}    
\ifnum \value{are-there-sections}=0 {%
  \newtheorem{proclaim}{Theorem}
} 
\else {%
  \newtheorem{proclaim}{Theorem}[section]
} 
\fi
\newtheorem{thm}[proclaim]{Theorem}

\newtheorem{lem}[proclaim]{Lemma}

\theoremstyle{bozont-sc}
\newtheorem{proclaim-special}[proclaim]{\specialthmname}

\theoremstyle{bozont-remark}

\newtheorem{notation}[proclaim]{Notation}

\newtheorem*{SubHeading*}{\SubHeadingName}%
\newtheorem{SubHeading}[proclaim]{\SubHeadingName}
\newtheorem{sSubHeading}[equation]{\sSubHeadingName}
\newenvironment{demo-r}[1]{\def\SubHeadingName{#1}\begin{SubHeading-r}}
  {\end{SubHeading-r}}%
\newenvironment{subdemo-r}[1]{\def\sSubHeadingName{#1}\begin{sSubHeading-r}}
  {\end{sSubHeading-r}} %
\newenvironment{demo*}[1]{\def\SubHeadingName{#1}\begin{SubHeading*}}
  {\end{SubHeading*}}%

\newtheorem{defn-thm}[proclaim]{Definition--Theorem}  

\theoremstyle{bozont-def}

\theoremstyle{bozont-reverse}    

\theoremstyle{bozont-reverse-sc}
\newtheorem{proclaimr-special}[proclaim]{\specialthmname}
{\def\specialthmname{#1}\begin{proclaimr-special}}%
{\end{proclaimr-special}}
\theoremstyle{bozont-remark-reverse}

\newtheorem{SubHeading-r}[proclaim]{\SubHeadingName}
\newtheorem{sSubHeading-r}[equation]{\sSubHeadingName}
\newtheorem{SubHeadingr}[proclaim]{\SubHeadingName}

\theoremstyle{bozont-def-newnum-reverse}    

\theoremstyle{bozont-def-reverse}

\newtheorem{newnumspecial}[proclaim]{\specialnewnumname}

\numberwithin{equation}{proclaim}
\numberwithin{figure}{section}

\newenvironment{enumerate-p}{
  \begin{enumerate}}
  {\setcounter{equation}{\value{enumi}}\end{enumerate}}
\newenvironment{enumerate-cont}{
  \begin{enumerate}
    {\setcounter{enumi}{\value{equation}}}}
  {\setcounter{equation}{\value{enumi}}
  \end{enumerate}}

\newlength{\swidth}
\makeatother

\DeclareMathAlphabet{\smallchanc}{OT1}{pzc}%
                                 {m}{it}
\DeclareFontFamily{OT1}{pzc}{}
\DeclareFontShape{OT1}{pzc}{m}{it}%
             {<-> s * [1.100] pzcmi7t}{}
\DeclareMathAlphabet{\mathchanc}{OT1}{pzc}%
                                 {m}{it}

\DeclareFontFamily{OMS}{rsfs}{\skewchar\font'60}
\DeclareFontShape{OMS}{rsfs}{m}{n}{<-5>rsfs5 <5-7>rsfs7 <7->rsfs10 }{}
\DeclareSymbolFont{rsfs}{OMS}{rsfs}{m}{n}
\DeclareSymbolFontAlphabet{\scr}{rsfs}

\newcommand{\sE}{\scr{E}}
\newcommand{\sF}{\scr{F}}
\newcommand{\sG}{\scr{G}}
\newcommand{\sH}{\scr{H}}

\newcommand{\sL}{\scr{L}}
\newcommand{\sM}{\scr{M}}

\newcommand{\sO}{\scr{O}}
\newcommand{\sP}{\scr{P}}
\newcommand{\sQ}{\scr{Q}}

\newcommand{\bP}{\mathbb{P}}

\newcommand{\factor}[2]{\left. \raise 2pt\hbox{\ensuremath{#1}} \right/
        \hskip -2pt\raise -2pt\hbox{\ensuremath{#2}}}

\def\coh#1.#2.#3.{H^{#1}(#2,#3)}
\def\dimcoh#1.#2.#3.{h^{#1}(#2,#3)}
\def\hypcoh#1.#2.#3.{\mathbb H_{\vphantom{l}}^{#1}(#2,#3)}
\def\loccoh#1.#2.#3.#4.{H^{#1}_{#2}(#3,#4)}
\def\dimloccoh#1.#2.#3.#4.{h^{#1}_{#2}(#3,#4)}
\def\lochypcoh#1.#2.#3.#4.{\mathbb H^{#1}_{#2}(#3,#4)}
\def\ses#1.#2.#3.{0  \longrightarrow  #1   \longrightarrow 
 #2 \longrightarrow #3 \longrightarrow 0} 
\def\sesshort#1.#2.#3.{0
 \rightarrow #1 \rightarrow #2 \rightarrow #3 \rightarrow 0}
\def\dist#1.#2.#3.{  #1   \longrightarrow 
 #2 \longrightarrow #3 \stackrel{+1}{\longrightarrow} } 
\def\CDdist#1.#2.#3.{  #1   @>>>  #2  @>>>   #3 @>+1>> }  
\def\shortses#1.#2.#3.{0  \rightarrow  #1   \rightarrow 
 #2  \rightarrow   #3 \rightarrow  0}
\def\shortdist#1.#2.#3.{  #1   \rightarrow 
 #2  \rightarrow   #3 \stackrel{+1}{\rightarrow} }  
\def\ddist#1.#2.#3.#4.#5.#6.{\CD
#1 @>>> #2 @>>> #3 @>+1>> \\
@VVV @VVV @VVV \\
#4 @>>> #5 @>>> #6 @>+1>> 
\endCD}
\def\ddistun#1.#2.#3.#4.#5.#6.{\CD
#1 @>>> #2 @>>> #3 @>+1>> \\
@. @VVV @VVV  \\
#4 @>>> #5 @>>> #6 @>+1>> 
\endCD}
\def\Iff#1#2#3{
\hfil\hbox{\hsize =#1
\vtop{\noin #2}
\hskip.5cm 
\lower.5\baselineskip\hbox{$\Leftrightarrow$}\hskip.5cm
\vtop{\noin #3}}\hfil\medskip}
\newcommand{\union}\cup
\newcommand{\intersect}\cap
\newcommand{\Union}\bigcup
\newcommand{\Intersect}\bigcap
\def\myoplus#1.#2.{\underset #1 \to {\overset #2 \to \oplus}}

\title{Viehweg's hyperbolicity conjecture is true over compact bases}
\author{Zsolt Patakfalvi}

\begin{document}

\maketitle

\section{Introduction}

Generalizing a classical conjecture of Shafarevich, Viehweg conjectured that manifolds mapping quasi-finitely to the moduli stack of canonically polarized manifolds are of log-general type. In fact, he conjectured more generally that if a manifold $U$ is a base of a family of projective manifolds with semi-ample canonical sheaves and maximal variation then $U$ is of log-general type. This is referred to usually as Viehweg's hyperbolicity conjecture, and is part of a larger package, generalizing different aspects of the aforementioned Shafarevich conjecture. We refer to \cite[Chapter 16]{HCD_KSJ_RA} for a detailed list of the related results and conjectures. 

This short paper proves Viehweg's hyperbolicity conjecture when $U$ is projective. We also show that the conjecture holds when the compactification of $U$ is not uniruled. The paper, at least in spirit, is the continuation of the very short paper \cite{KS_KSJ_FOV}. That  paper proves Viehweg's hyperbolicity conjecture over compact bases assuming the full Minimal Model Program and the Abundance conjecture. Here we manage to drop these two assumptions.

Let us introduce first the basic setup used in the article.

\begin{notation}
\label{notation:Viehweg1}
Fix a  a projective  manifold $B$ over an algebraically closed field of characteristic zero,  a normal crossing divisor $\Delta \subseteq B$, and define $U := B \setminus \Delta$. 
\end{notation}

Recall that for a family $X \to U$ of varieties over an integral base, the \emph{variation is maximal}, if for a generic $u \in U$ there are finitely many $u' \in U $ such that $X_u$ is birational to $X_{u'}$ \cite[Definition 2.8]{KJ_SOT}. Many times we will need to extend Notation \ref{notation:Viehweg1} as follows.

\begin{notation}
\label{notation:Viehweg2}
In addition to Notation \ref{notation:Viehweg1}, assume that there is a family $X \to U$ of smooth projective manifolds with maximal variation and $\omega_{X/U}$ $f$-semi-ample.
\end{notation}

Note that a particular case of Notation \ref{notation:Viehweg2} is when  $U$ maps quasi-fintely to  the moduli stack of canonically polarized manifolds. This might be the primary case of interest for some of the readers. 

The main result of the paper is as follows.

\begin{thm}
\label{thm:Viehweg}
Viehweg's hyperbolicity conjecture is true over compact or non-uniruled bases. That is, in the situation of Notation \ref{notation:Viehweg2}, if either
\begin{enumerate}
 \item \label{itm:Viehweg:compact} $\Delta= \emptyset$ or
\item \label{itm:Viehweg:uniruled} $B$ is not uniruled
\end{enumerate}
 then $\omega_B(\Delta)$ is big. 
\end{thm}


\section{Technicalities}

First, recall the following fundamental property of base-spaces of manifolds with semi-ample canonical sheaves. It is the main ingredient in the proofs of \cite{KS_KSJ_FOV} and \cite{KS_KSJ_TS} as well.

\begin{lem} \cite[Theorem 1.4]{VE_ZK_BS}
\label{lem:Viehweg-Zuo}
In the situation of Notation \ref{notation:Viehweg2},  there is a big sheaf $\sF$ contained in $\Omega_B(\log \Delta)^{\otimes m}$ for some integer $m>0$.
\end{lem}

Second, we list and prove the following two well-known facts.

\begin{lem}
\label{lem:determinant}
For any vector bundle $\sE$, $\det (\sE^{\otimes m}) \cong (\det (\sE))^{\otimes N}$ for some positive integer $N>0$.
\end{lem}

\begin{proof}
Consider the representation $\det ((\_)^{\otimes m})$. It is a one dimensional representation of the general linear group, hence it is $\det(\_)^N$ for some integer $N$. One can see that $N$ is positive here, by plugging in $\sE := \sO_{\bP^n}(1)^{\oplus r}$.
\end{proof}

By \emph{effective line bundle} we mean a line bundle corresponding to an effective divisor.

\begin{lem}
\label{lem:pseff}
If $\sL$ is a pseudo-effective and $\sM$ a big (resp.  effective) line bundle on a projective manifold, then $\sL \otimes \sM$ is big (resp. pseudo-effective).
\end{lem}

\begin{proof}
By \cite[Theorem 1.2]{DJP_OTG} the effective cone is the closure of the cone generated by the classes of effective divisors, and its interior is the cone generated by big divisors. Then the statement follows.
\end{proof}

We end this section with the other main ingredient of our proof, a positivity property for log-cotangent bundles of pairs.

\begin{lem}
\label{lem:factor}
Using Notation \ref{notation:Viehweg1}, if $B$ is not uniruled and $ \psi : \Omega_B(\log \Delta)^{\otimes m} \to \sQ$ is a torsion free quotient, then $\det \sQ$ is pseudo-effective.
\end{lem}

\begin{proof}
Consider $\Omega_B^{\otimes m}$ as a subsheaf of $\Omega_B(\log \Delta)^{\otimes m}$ and define $\sP:= \psi (\Omega_B^{\otimes m})$. Then $\sP$ is torsion-free as well, and generically isomorphic subsheaf of $\sQ$. In particular, then $\det \sQ = (\det \sP) \otimes \sM$ for some effective line bundle $\sM$. By \cite[Theorem 0.3]{CF_PT_GSO}, $\det \sP$ is pseudo-effective. Hence By Lemma \ref{lem:pseff}, so is $\det \sQ$.
\end{proof}

\section{The proof of Theorem \ref{thm:Viehweg}}

\begin{proof}[Proof of Theorem \ref{thm:Viehweg}]
Consider the big sheaf $\sF  \subseteq \Omega_B(\log \Delta)^{\otimes m}$ guaranteed by Lemma \ref{lem:Viehweg-Zuo}. Let $\sG$ be the saturation of $\sF$ in $\Omega_B(\log \Delta)^{\otimes m}$ and define $\sH := \factor{ \Omega_B(\log \Delta)^{\otimes m}}{\sG}$. In the case of assumption \eqref{itm:Viehweg:compact}, $B$ is not uniruled by \cite[Thm. 1]{KSJ_SF}, in the other case, this is the assumption itself. Note also that
$\sH$ is torsion free since $\sG$ is saturated. Hence, by Lemma \ref{lem:factor}, $\det \sH$ is pseudo-effective. So, there is a positive integer $N$ such that
\begin{equation*} 
 \omega_B(\Delta)^N \cong 
\underbrace{\det ( \Omega_B(\log \Delta)^{\otimes m})}_{\textrm{Lemma \ref{lem:determinant}}} \cong 
\underbrace{\det \sG}_{\textrm{big}} \otimes \underbrace{\det \sH}_{\textrm{pseudo-effective}} .
\end{equation*}
So, by Lemma \ref{lem:pseff}, $\omega_B(\Delta)$ is big. This finishes our proof.
\end{proof}


\bibliographystyle{skalpha}
\bibliography{include}

\end{document}